\documentclass[12pt,a4paper]{amsart}

\newtheorem{theo+}              {Theorem}           [section]
\newtheorem{prop+}  [theo+]     {Proposition}
\newtheorem{coro+}  [theo+]     {Corollary}
\newtheorem{lemm+}  [theo+]     {Lemma}
\newtheorem{exam+}  [theo+]     {Example}
\newtheorem{rema+}  [theo+]     {Remark}
\newtheorem{defi+}  [theo+]     {Definition}
\def \r{\mbox{${\mathbb R}$}}

\newenvironment{theorem}{\begin{theo+}}{\end{theo+}}

\newenvironment{corollary}{\begin{coro+}}{\end{coro+}}
\newenvironment{lemma}{\begin{lemm+}}{\end{lemm+}}

\usepackage{amsthm}
\theoremstyle{plain} \theoremstyle{remark}
\newtheorem{remark}{Remark}
\newtheorem{example}{Example}

\def\E{/\kern-1.0em \equiv }

\title{Biharmonic maps from tori into a 2-sphere}
\author{Ze-Ping Wang }
\author{Ye-Lin Ou}
\author{Han-Chun Yang$^{*}$ }
\address{Department of Mathematics, \newline\indent Yunnan University,\newline\indent
Kunming 650091, P. R. China
\newline\indent E-mail:zeping.wang@gmail.com \;(Wang)\\\newline\indent E-mail: hyang@ynu.edu.cn\; (H. Yang)\\\newline\indent
\newline\indent Department of
Mathematics,\newline\indent Texas A $\&$ M
University-Commerce,\newline\indent Commerce, TX 75429
USA.\newline\indent E-mail:yelin.ou@tamuc.edu\; (Ou)}
\thanks{ *Research supported by the NSF of China (11361073)}
\begin{document}
\title[Biharmonic maps from tori into a 2-sphere]{Biharmonic maps from tori into a 2-sphere}
\date{06/18/2014}
\maketitle
\section*{Abstract}
\begin{quote}
{\footnotesize Biharmonic maps are generalizations of harmonic maps. A well-known result of Eells and Wood on harmonic maps between surfaces shows that there exists no
harmonic map from a torus into a sphere (whatever the metrics chosen) in the homotopy class of maps of Brower degree $\pm 1$. It would be interesting to know if there exists any biharmonic map in that homotopy class of maps. In this paper, we obtain some classifications on biharmonic maps from a torus into a sphere, where the torus is provided with a flat or a class of non-flat metrics whilst the sphere is provided with the standard metric. Our results show that there exists no proper biharmonic maps of degree $\pm 1$ in a large family of maps from a torus into a sphere. 
}
\end{quote}
\section{Introduction}
All objects including manifolds, tensor fields, and maps studied in this paper are supposed to be smooth.\\

{\em Harmonic maps} are maps $\varphi: (M,g)\longrightarrow (N,h)$ between Riemannian manifolds that minimize the energy functional
\begin{equation}\notag
E(\varphi)=\frac{1}{2}\int_\Omega |{\rm d}\varphi|^2v_g,
\end{equation}
where $\Omega$ is a compact domain of $ M$. Analytically, a harmonic map is a solution of a system of 2nd order PDEs
\begin{equation}\label{fhm}
\tau(\varphi) \equiv {\rm Tr}_g\nabla\,d \varphi=g^{ij}(\varphi^{\sigma}_{ij}-\Gamma^k_{ij}\varphi^{\sigma}_{k}+\bar{\Gamma}^{\sigma}_{\alpha\beta}\varphi^{\alpha}_{i}\varphi^{\beta}_{j})\frac{\partial}{\partial y^{\sigma}}
=0,
\end{equation}
where $\tau(\varphi)$ is called the tension field of
the map $\varphi$.\\

{\em Biharmonic maps} are generalizations of harmonic maps, which are maps\\ $\varphi: (M,g)\longrightarrow (N,h)$ between Riemannian manifolds that are critical points of the bienergy functional defined by
\begin{equation}\notag
E_{2}(\varphi)=\frac{1}{2}\int_\Omega |\tau(\varphi)|^2v_g,
\end{equation}
where $\Omega$ is a compact domain of $ M$. Biharmonic map equation is a system of 4th order nonlinear PDEs (\cite{Ji1})
\begin{equation}\label{BTF}
\tau_{2}(\varphi):={\rm
Trace}_{g}(\nabla^{\varphi}\nabla^{\varphi}-\nabla^{\varphi}_{\nabla^{M}})\tau(\varphi)
- {\rm Trace}_{g} R^{N}({\rm d}\varphi, \tau(\varphi)){\rm d}\varphi =0,
\end{equation}
where $R^{N}$ denotes the curvature operator of $(N, h)$ defined by
$$R^{N}(X,Y)Z=
[\nabla^{N}_{X},\nabla^{N}_{Y}]Z-\nabla^{N}_{[X,Y]}Z.$$

As a harmonic map is always a biharmonic map, we call a biharmonic map that is not harmonic a {\em proper biharmonic map}. \\

Since 2000, the study of biharmonic maps has been attracting growing interest of many mathematicians and it has become a popular topic of research with many interesting results. For some recent geometric study of general biharmonic maps, we refer the readers to \cite{BK}, \cite{BFO2}, \cite{BMO2}, \cite{LOn1}, \cite{MO}, \cite{NUG}, \cite{Ou1}, \cite{Ou4}, \cite{OL}, \cite{Oua}, \cite{WOY} and the references therein. For some recent progress on biharmonic submanifolds (i.e., submanifolds whose defining isometric immersions are biharmonic maps) see \cite{Ch1}, \cite{Ch2}, \cite{Ch3}, \cite{CI}, \cite{Ji2}, \cite{Ji3}, \cite{Di}, \cite{HV}, \cite{CMO1}, \cite{CMO2}, \cite{BMO1}, \cite{BMO3}, \cite{Ou3}, \cite{OT}, \cite{OW}, \cite{NU}, \cite{TO}, \cite{CM}, \cite{AGR}, \cite{Wh} and the references therein. For biharmonic conformal immersions and submersions see \cite{Ou2}, \cite{Ou5}, \cite{BFO1}, \cite{LO}, \cite{WO} and the references therein.\\

For harmonic maps between surfaces, a very interesting result proved by Eells and Wood in \cite{EW} states that {\bf there exists no
harmonic map from a torus $T^2$ into a sphere $S^2$ (whatever the metrics chosen) in the homotopy class of maps of Brower degree $\pm 1$}. It would be interesting to know if there exists any proper biharmonic map from a torus $T^2$ into a sphere $S^2$ (whatever the metrics chosen) in the homotopy class of maps of Brower degree $\pm 1$. Motivated by this, we study biharmonic maps from a torus into a sphere in this paper. We are able to obtain some classifications of proper biharmonic maps in a large family of maps from $T^2$ into $S^2$ which include the Gauss map of the torus $T^2\longrightarrow \r^3$ and the compositions $T^2\longrightarrow S^3\longrightarrow S^2$ of some immersions of $T^2$ into $S^3$ followed by the Hopf fibration. Here, the torus is provided with a flat or a class of non-flat metrics whilst the sphere is provided with the standard metric (Theorem \ref{T2S2} and Theorem \ref{NTS}).\\
\section{Constructions of maps from a torus into a $2$-sphere} 
In order to study biharmonic maps from a torus $T^2$ into a sphere $S^2$ we need to have a good source of maps from a torus into a sphere. In this section, we present three ways to construct maps from $T^2$ into $S^2$.
\begin{itemize}
\item [1.] {\em Construction via Hopf fibration}. For any map $f: T^2\longrightarrow S^3$, we have a map from torus into $2$-sphere, $H\circ f:T^2\longrightarrow S^2$, where $H:S^3\longrightarrow S^2$ is the Hopf fibration. Here, we view the Hopf fibration as the restriction of the Hopf construction of the standard multiplication of complex numbers, i.e., $H: \mathbb{C}\times \mathbb{C}\longrightarrow \r\times \mathbb{C}$ with $H(z, w)=(|z|^2-|w|^2, 2z\bar{w})$.
\item [2.] {\em Construction via Radial Projection}. For any map $f: T^2\longrightarrow \r^3\setminus\{0\}$, we have a map from torus into $2$-sphere, $P\circ f:T^2\longrightarrow S^2$, where $P:\r^3\setminus\{0\}\longrightarrow S^2$ with $P(x)=\frac{x}{|x|}$ is the radial projection from $\r^3\setminus\{0\}$ onto $S^2$.
\item [3.] {\em Construction via Gauss map of a torus}. It is well known that if $f: T^2\longrightarrow \r^3$ is an immersion of a torus into $\r^3$, then the Gauss map gives a map from the torus into a $2$-sphere defined by $G:T^2\longrightarrow S^2$, $G(x)=n(x)$ with $n(x)$ being the unit normal vector at the point $x\in T^2$.
\end{itemize}
\begin{example}
Let $f: T^2\longrightarrow S^3$ with $f(x,y)=(\cos (kx) e^{imy} , \sin( kx)e^{iny} )$ be a family of immersions studied by Lawson in \cite{La} . Then, the composition $\varphi=H\circ f: T^2\longrightarrow S^2$ gives a family of maps from a torus into a $2$-sphere defined by
\begin{eqnarray}\notag
\varphi(x,y)&=&\big(\cos^2 (kx)-\sin^2(kx) ,\; 2\cos (kx)\sin (kx) e^{i(m-n)y}\big)\\
&=&\big( \cos (2kx),\; \sin (2kx) e^{i(m-n)y} \big).
\end{eqnarray}
If we use geodesic polar coordinates $(\rho, \phi)$ on the $2$-sphere, then this family of maps can be represented as:
\begin{eqnarray}\notag
\varphi=H\circ f:&&T^2\longrightarrow (S^2, d\rho^2+\sin^2\rho d\phi^2), \; \varphi(x,y)=(\rho(x,y),\phi(x,y)),\; {\rm with}\\\label{Map1}
&& \begin{cases}
\rho(x,y)=2kx, \\\phi(x,y)=(m-n)y.
\end{cases}
\end{eqnarray}
\end{example}
\begin{example}
Let $f: S^1\times S^1\longrightarrow \r^4$ be a family of immersions of flat tori into $\r^4$ defined by
$f(x,y)=(Ae^{ix} , B e^{iy} )$ with constants $A, B$ satisfying $A^2+B^2\ne 0$. Postcomposing this map with the radial projection $P:\r^4\longrightarrow S^3, P(x)=\frac{x}{|x|}$ we have a family of map $F=P\circ f: S^1\times S^1\longrightarrow S^3$ with
$F(x,y)=(\frac{A}{\sqrt{A^2+B^2}}e^{ix} , \frac{B}{\sqrt{A^2+B^2}} e^{iy} )$. If we denote $\frac{A}{\sqrt{A^2+B^2}}=\cos s$, then $\frac{B}{\sqrt{A^2+B^2}}=\sin s$, then the family of the maps can be written as $F_s: T^2\longrightarrow S^3$ with
$F_s(x,y)=(e^{ix} \cos s , e^{iy} \sin s )$. Applying construction via Hopf fibration we have a map from $T^2$ into $S^2$:
\begin{eqnarray}\notag
\varphi=H\circ f:&&T^2\longrightarrow (S^2, d\rho^2+\sin^2\rho d\phi^2), \varphi(x,y)=(\rho(x,y), \phi(x,y)), {\rm with}\\\label{Map2}
&& \begin{cases}
\rho(x,y)=2s,\; s={\rm constant}, \\\phi(x,y)=x-y,
\end{cases}
\end{eqnarray}
where $(\rho, \phi)$ are the geodesic polar coordinates on $S^2$.
\end{example}
\begin{remark}
As it was observed in \cite{BW} (Example 3.3.18) that except $s=0, \pi/2$, the family of maps $F_s$ are embeddings of tori into $S^3$ and all these maps are harmonic maps with constant energy density called eigenmaps.
\end{remark}
\begin{example}
Let $f: S^1\times S^1\longrightarrow S^3$ with $f(x,y)=(\sqrt{1-\beta^2(x)} \;e^{ix} , \beta(x)e^{iy} )$, where $\beta(x)$ is a smooth function taking value between $0$ and $1$, be the family of immersions of flat tori into $S^3$ studied by Brendle in \cite{Br}. Introducing the new variable by denoting $\cos \alpha (x)=\sqrt{1-\beta^2(x)}$, we have $f(x,y)=(e^{ix}\cos \alpha(x ) , e^{iy}\sin \alpha(x) )$, Postcomposing this map with the Hopf fibration $H:S^3\longrightarrow S^2$, we have a family of map $F=H\circ f: S^1\times S^1\longrightarrow S^3$ with
$F(x,y)=(\cos 2\alpha(x), e^{i(x-y)}\sin 2\alpha(x) )$. If we use the geodesic polar coordinates $(\rho, \phi)$ on the $2$-sphere, then, this family of maps can be represented as:
\begin{eqnarray}\notag
\varphi=H\circ f:&&T^2\longrightarrow (S^2, d\rho^2+\sin^2\rho d\phi^2), \varphi(x,y)=(\rho(x,y), \phi(x,y)), {\rm with}\\\label{Map3}
&& \begin{cases}
\rho(x,y)=2\alpha(x), \\\phi(x,y)=x-y.
\end{cases}
\end{eqnarray}
\end{example}
\begin{example}\label{Ex4}
It was proved in \cite{Ou4} that the map $f:S^1\times S^1\longrightarrow S^3$ with
\begin{eqnarray}
f(x,y)=&&\big(\cos^2\frac{x+y}{2},\;\;\sin \frac{x+y}{2}\,\cos
\frac{x+y}{2},\,-\sin \frac{x+y}{2}\,\sin \frac{x-y}{2},\\\notag
&&-\sin \frac{x+y}{2}\,\cos \frac{x-y}{2}\big),
\end{eqnarray}
is a biharmonic map from a flat torus into the $3$-sphere. It can be easily checked that, up to an isometry, the map can be written as\\
$f(x,y)=\big(e^{i\frac{x+y}{2}}\cos\frac{-(x+y)}{2} ,\;e^{i\frac{x-y}{2}}\sin \frac{-(x+y)}{2}\big)$. Applying construction via Hopf fibration we have a map from $T^2$ into $S^2$:
\begin{eqnarray}\notag
\varphi=H\circ f:&&T^2\longrightarrow (S^2, d\rho^2+\sin^2\rho d\phi^2), \varphi(x,y)=(\rho(x,y), \phi(x,y)), {\rm with}\\\label{Map4}
&& \begin{cases}
\rho(x,y)=-x-y, \\\phi(x,y)=y,
\end{cases}
\end{eqnarray}
where $(\rho, \phi)$ are geodesic polar coordinates in $S^2$.
\end{example}
\begin{example}\label{Ex5}
Let $X: T^2 \longrightarrow \r^3$ be the standard embedding $X(r, \theta)=( a\sin r,\; (b+a\cos r)\cos \theta,\; (b+a\cos r)\sin \theta)$ of a torus into $\r^3$. A straightforward computation gives the induced metric and the unit normal vector field of the torus to be
\begin{equation}\notag
g_{T}=a^2dr^2+(b+a\cos r)^2d\theta^2,
\end{equation}
and
\begin{eqnarray}\notag
n=\frac{X_r\times X_\theta }{|X_r\times X_\theta|}= ( \sin r,\; \cos r\cos \theta,\; \cos r\sin \theta).
\end{eqnarray}
respectively. If we use the geodesic polar coordinates $(\rho, \phi)$ on the unit sphere so that a generic point $(x, y, z)\in S^2$ is represented as $(x, y, z)=(\sin \rho, e^{i\phi}\cos \rho )$, then, the Gauss map of the torus can be written as
\begin{eqnarray}\notag
\varphi: &&(T^2, a^2dr^2+(b+a\cos r)^2d\theta^2) \longrightarrow (S^2, d\rho^2+\cos^2 \rho d\phi^2)\\\notag &&\varphi(r,\theta)=(\rho(r, \theta), \phi(r,\theta)) \;{\rm with}\\\label{Map5}
&& \begin{cases}
\rho(r, \theta)=r, \\\phi(r, \theta)=\theta.
\end{cases}
\end{eqnarray}
\end{example}
\begin{example}
Let $X: T^2 \longrightarrow \r^3$ be the standard embedding $X(r, \theta)=( a\sin r,\; (b+a\cos r)\cos \theta,\; (b+a\cos r)\sin \theta)$ of a torus into $\r^3$. Using the construction via radial projection $P:\r^3\setminus\{0\}\longrightarrow S^2$ with $P(x)=\frac{x}{|x|}$, we have a family of maps from tori into a $2$-sphere given by $\varphi=P\circ X:T^2\longrightarrow S^2$, with
\begin{eqnarray}
\varphi (r, \theta)=\big( \cos \alpha (r),\; \sin\alpha(r)\cos \theta,\; \sin\alpha(r)\sin \theta\big),
\end{eqnarray}
where $\arccos \alpha(r)=\frac{a\sin r}{\sqrt{a^2+b^2+2ab\cos r}}$.
Again, with respect to the geodesic polar coordinates $(\rho, \phi)$ on the unit sphere so that a generic point $(x, y, z)\in S^2$ is represented as $(x, y, z)=(\cos \rho, e^{i\phi}\sin \rho )$, then, this family of maps from tori into $S^2$ can be written as:
\begin{eqnarray}\notag
\varphi: &&(T^2, a^2dr^2+(b+a\cos r)^2d\theta^2) \longrightarrow (S^2, d\rho^2+\sin^2 \rho d\phi^2)\\\notag &&\varphi(r,\theta)=(\rho(r, \theta), \phi(r,\theta)) \;{\rm with}\\\label{Map6}
&& \begin{cases}
\rho(r, \theta)=\alpha (r) , \\\phi(r, \theta)=\theta,
\end{cases}
\end{eqnarray}
which are rotationally symmetric maps.
\end{example}

The 6 families of maps given in Examples 1-6 lead us to study the biharmonicity of the following family of maps
\begin{eqnarray}\label{MF}
\varphi: T^2\longrightarrow (S^2, d\rho^2+\sin^2\rho d\phi^2),\\\notag
\varphi (r, \theta)=(ar+b\theta+c, mr+n\theta+l).
\end{eqnarray}

Clearly, the family of maps defined by (\ref{MF}) includes families of maps defined in (\ref{Map1}), (\ref{Map2}), (\ref{Map4}), (\ref{Map5}), and parts of the families given in (\ref{Map3}) and (\ref{Map6}). Our main results in this paper include a complete classifications of proper biharmonic maps in the family of maps defined by (\ref{MF}), where the torus is provided with a flat or a non-flat metric whilst the sphere is provided with the standard metric (Theorem \ref{T2S2} and Theorem \ref{NTS}).
\section{Biharmonic maps from a flat torus into a $2$-sphere}
\begin{lemma}\cite{OL}\label{OL}
Let $\varphi :(M^{m}, g)\longrightarrow (N^{n}, h)$ be a map between
Riemannian manifolds with $\varphi (x^{1},\ldots, x^{m})=(\varphi^{1}(x),
\ldots, \varphi^{n}(x))$ with respect to local coordinates $(x^{i})$ in
$M$ and $(y^{\alpha})$ in $N$. Then, $\varphi$ is biharmonic if and
only if it is a solution of the following system of PDE's
\begin{eqnarray}\notag\label{BI3}
&& \Delta\tau^{\sigma} +2g(\nabla\tau^{\alpha},
\nabla\varphi^{\beta}) {\bar \Gamma_{\alpha\beta}^{\sigma}}
+\tau^{\alpha}\Delta \varphi ^{\beta}{\bar
\Gamma_{\alpha\beta}^{\sigma}}\\ && +\tau^{\alpha}
g(\nabla\varphi^{\beta}, \nabla\varphi^{\rho})(\partial_{\rho}{\bar
\Gamma_{\alpha\beta}^{\sigma}}+{\bar
\Gamma_{\alpha\beta}^{\nu}}{\bar \Gamma_{\nu\rho}^{\sigma}})
-\tau^{\nu}g(\nabla\varphi^{\alpha}, \nabla\varphi^{\beta}){\bar
R}_{\beta\,\alpha \nu}^{\sigma}=0,\\\notag && \sigma=1,\, 2,\,
\ldots, n,
\end{eqnarray}
where $\tau^1, \ldots, \tau^n$ are components of the tension field
of the map $\varphi$, $\nabla,\;\Delta$ denote the gradient and the
Laplace operators defined by the metric $g$, and
${\bar\Gamma_{\alpha\beta}^{\sigma}}$ and ${\bar R}_{\beta\,\alpha
\nu}^{\sigma}$ are the components of the connection and the
curvature of the target manifold.
\end{lemma}
In order to prove our classification theorem we need the following lemma.
\begin{lemma}\label{PL}
Let $\varphi:(M^2,dr^2+\sigma^2(r)d\theta^2)\longrightarrow
(N^2,d\rho^2+\lambda^2(\rho)d\phi^2)$ with $\varphi(r,\theta)=(ar+b\theta+c, mr+n\theta+l)$
. Then, $\varphi$ is biharmonic if and only if it solves the system
\begin{equation}\label{p1}
\begin{cases}
\frac{x_{\theta\theta}}{\sigma^2}+x_{rr}+\frac{\sigma'}{\sigma}x_r-(m^2+\frac{n^2}{\sigma^2})\left(\lambda\lambda'(\rho)\right)'(\rho)
x-\left(2my_r+\frac{2n}{\sigma^2}y_\theta+y^2\right)\lambda\lambda'(\rho)=0,\\
\frac{y_{\theta\theta}}{\sigma^2}+y_{rr}+\frac{\sigma'}{\sigma}y_{r}+2\left(ay_{r}
+mx_r+\frac{by_\theta}{\sigma^2}
+\frac{nx_\theta}{\sigma^2}\right)\frac{\lambda'(\rho)}{\lambda}\\
+2\left(\frac{m\sigma'\lambda'(\rho)}{\sigma\lambda}
+(am+\frac{bn}{\sigma^2})\frac{(\lambda\lambda'(\rho))'(\rho)}{\lambda^2}\right)x=0,\\
x=\tau^1=a\frac{\sigma'}{\sigma}-(m^2+\frac{n^2}{\sigma^2})\lambda\lambda'(\rho),\\
y=\tau^2=m\frac{\sigma'}{\sigma}+2(am+\frac{bn}{\sigma^2})\frac{\lambda'(\rho)}{\lambda}.
\end{cases}
\end{equation}
\end{lemma}
\begin{proof}
One can easily compute the connection coefficients of the domain and
the target surfaces to get
\begin{eqnarray}
\Gamma^1_{11}=0, \hskip0.3cm\Gamma^1_{12}=0, \hskip0.3cm
\Gamma^1_{22}=-\sigma\sigma',\hskip0.3cm \Gamma^2_{11}=0, \hskip0.3cm
\Gamma^2_{12}=\frac{\sigma'}{\sigma}, \hskip0.3cm \Gamma^2_{22}=0,
\end{eqnarray}
and
\begin{eqnarray}
\hskip0.6cm\bar{\Gamma}^1_{11}=0, \hskip0.3cm\bar{\Gamma}^1_{12}=0,
\hskip0.3cm \bar{\Gamma}^1_{22}=-\lambda\lambda'(\rho),\hskip0.3cm
\bar{\Gamma}^2_{11}=0, \hskip0.3cm \bar{\Gamma}^2_{12}=\frac{\lambda'(\rho)}{\lambda},
\hskip0.3cm \bar{\Gamma}^2_{22}=0.
\end{eqnarray}
A further computation gives the components of the Riemannian curvature of
the target surface as:
\begin{eqnarray}
\hskip0.6cm\bar{R}^1_{221}=-\bar{R}^1_{212}=\lambda\lambda''(\rho),\;
\bar{R}^2_{112}=-\bar{R}^2_{121}=\frac{\lambda''(\rho)}{\lambda}, \; {\rm
all\;other}\; \quad \bar{R}^l_{kij}=0.
\end{eqnarray}
We compute the components of tension field of the map $\varphi$ to have
\begin{eqnarray}\notag
\tau^1 &=& g^{ij}(\varphi^{1}_{ij}-\Gamma^k_{ij}\varphi^{1}_{k}+\bar{\Gamma}^1_{\alpha\beta}\varphi^{\alpha}_{i}\varphi^{\beta}_{j})
=-\Gamma_{11}^k\varphi^1_k-\frac{1}{\sigma^2}\Gamma_{22}^k\varphi^1_k+g(\nabla\varphi^2, \nabla\varphi^2)\bar{\Gamma}^1_{22}\\\label{pl3}
&=&-m^2\lambda\lambda'(\rho)+\frac{a\sigma\sigma'-n^2\lambda\lambda'(\rho)}{\sigma^2},\\\notag
\tau^2&=&g^{ij}(\varphi^{2}_{ij}-\Gamma^k_{ij}\varphi^{2}_{k}+\bar{\Gamma}^2_{\alpha\beta}\varphi^{\alpha}_{i}\varphi^{\beta}_{j})=-\Gamma_{11}^k\varphi^2_k-\frac{1}{\sigma^2}\Gamma_{22}^k\varphi^2_k+2g(\nabla\varphi^1, \nabla\varphi^2)\bar{\Gamma}^2_{12}\\\label{pl4}
&=&m\frac{\sigma'}{\sigma}+2(am+\frac{bn}{\sigma^2})\frac{\lambda'(\rho)}{\lambda}.
\end{eqnarray}
Using notations $x=\tau^1$, $y=\tau^2$, $x_r=\frac{\partial x}{\partial r}, x_\theta=\frac{\partial x}{\partial \theta},
y_r=\frac{\partial y}{\partial r}, y_\theta=\frac{\partial y}{\partial \theta}, \\x_{r\theta}=\frac{\partial^{2} x}{\partial r\partial\theta}, y_{r\theta}=\frac{\partial^{2} y}{\partial r\partial\theta},x_{rr}=\frac{\partial^{2} x}{\partial r^2}, y_{rr}=\frac{\partial^{2} y}{\partial r^2},
x_{\theta\theta}=\frac{\partial^{2} x}{\partial \theta^2}, y_{\theta\theta}=\frac{\partial^{2} y}{\partial \theta^2}$, we compute
\begin{eqnarray}\label{12}
\Delta\tau^1=g^{ij}(\tau^1_{ij}-\Gamma^k_{ij}\tau^1_k)=x_{rr}+\frac{\sigma\sigma'x_r+x_{\theta\theta}}{\sigma^2},\\
\Delta\tau^2=g^{ij}(\tau^2_{ij}-\Gamma^k_{ij}\tau^2_k)=y_{rr}+\frac{\sigma\sigma'y_r+y_{\theta\theta}}{\sigma^2},\\\notag
\end{eqnarray}
\begin{eqnarray}
2g(\nabla\tau^{\alpha},\nabla\varphi^{\beta})\bar{\Gamma}^1_{\alpha\beta}=2g(\nabla\tau^{2},\nabla\varphi^{2})\bar{\Gamma}^1_{22}=-2\lambda\lambda'(\rho)(my_r+\frac{n}{\sigma^2}y_\theta),
\end{eqnarray}
\begin{eqnarray}
&&2g(\nabla\tau^{\alpha},\nabla\varphi^{\beta})\bar{\Gamma}^2_{\alpha\beta}=2g(\nabla\tau^{1},\nabla\varphi^{2})\bar{\Gamma}^2_{12}+2g(\nabla\tau^{2},\nabla\varphi^{1})\bar{\Gamma}^2_{21}\\\notag
&&=\frac{2(m x_r+\frac{n}{\sigma^2}x_\theta)\lambda'(\rho)}{\lambda}+\frac{2(a y_r+\frac{b}{\sigma^2}y_\theta)\lambda'(\rho)}{\lambda},\\\notag
\end{eqnarray}
\begin{eqnarray}
\tau^{\alpha}\Delta\varphi^{\beta}\bar{\Gamma}^1_{\alpha\beta}=\tau^{2}\Delta\varphi^{2}\bar{\Gamma}^1_{22}=-\frac{m\sigma'\lambda\lambda'(\rho)}{\sigma}y,\\\notag
\end{eqnarray}
\begin{eqnarray}
\tau^{\alpha}\Delta\varphi^{\beta}\bar{\Gamma}^2_{\alpha\beta}=\tau^{1}\Delta\varphi^{2}\bar{\Gamma}^2_{12}+\tau^{2}\Delta\varphi^{1}\bar{\Gamma}^2_{21}=\frac{\lambda'(\rho)}{\lambda}\left(\frac{m\sigma'}{\sigma}x+\frac{a\sigma'}{\sigma}y\right),\\\notag
\end{eqnarray}
\begin{eqnarray}
&&g(\nabla\varphi^{\beta}, \nabla\varphi^{\rho})\partial_{\rho}{\bar
\Gamma_{\alpha\beta}^{1}}=\tau^{\alpha}g(\nabla\varphi^{\beta},\nabla\bar{\Gamma}^1_{\alpha\beta})=\tau^{2}g(\nabla\varphi^{2},\nabla\bar{\Gamma}^1_{22})\\\notag
&&=-\left(am+\frac{bn}{\sigma^2}\right)\left(\lambda\lambda'(\rho)\right)'(\rho)y,
\end{eqnarray}
\begin{eqnarray}
&&g(\nabla\varphi^{\beta}, \nabla\varphi^{\rho})\partial_{\rho}{\bar
\Gamma_{\alpha\beta}^{2}}=\tau^{\alpha}g(\nabla\varphi^{\beta},\nabla\bar{\Gamma}^2_{\alpha\beta})\\\notag
&&=\tau^{1}g(\nabla\varphi^{2},\nabla\bar{\Gamma}^2_{12})+\tau^{2}g(\nabla\varphi^{1},\nabla\bar{\Gamma}^2_{21}),\\\notag
&&=\left(am+\frac{bn}{\sigma^2}\right)\left(\frac{\lambda'(\rho)}{\lambda}\right)'(\rho)x
+\left(a^2+\frac{b^2}{\sigma^2}\right)\left(\frac{\lambda'(\rho)}{\lambda}\right)'(\rho)y,\\\notag
\end{eqnarray}
\begin{eqnarray}
&&\tau^{\alpha}g(\nabla\varphi^{\beta},\nabla\varphi^{\rho})\bar{\Gamma}^v_{\alpha\beta}\bar{\Gamma}^1_{v\rho}=\tau^{1}g(\nabla\varphi^{2},\nabla\varphi^{2})
\bar{\Gamma}^2_{12}\bar{\Gamma}^1_{22}
+\tau^{2}g(\nabla\varphi^{1},\nabla\varphi^{2})
\bar{\Gamma}^2_{21}\bar{\Gamma}^1_{22}\\\notag
&&=-(m^2+\frac{n^2}{\sigma^2})\lambda'^2(\rho)x-(am+\frac{bn}{\sigma^2})\lambda'^2(\rho)y
\end{eqnarray}
\begin{eqnarray}
&&\tau^{\alpha}g(\nabla\varphi^{\beta},\nabla\varphi^{\rho})\bar{\Gamma}^v_{\alpha\beta}\bar{\Gamma}^2_{v\rho}\\\notag
&&=\tau^{1}g(\nabla\varphi^{2},\nabla\varphi^{1})
\bar{\Gamma}^2_{12}\bar{\Gamma}^2_{21}+\tau^{2}g(\nabla\varphi^{1},\nabla\varphi^{1})
\bar{\Gamma}^2_{21}\bar{\Gamma}^2_{21}+\tau^{2}g(\nabla\varphi^{2},\nabla\varphi^{2})
\bar{\Gamma}^1_{22}\bar{\Gamma}^2_{12}\\\notag
&&=(am+\frac{bn}{\sigma^2})(\frac{\lambda'(\rho)}{\lambda})^2x+(a^2+\frac{b^2}{\sigma^2})(\frac{\lambda'(\rho)}{\lambda})^2y
-(m^2+\frac{n^2}{\sigma^2})\lambda'^2(\rho)y,\\\notag
\end{eqnarray}
\begin{eqnarray}
&&-\tau^vg(\nabla\varphi^{\alpha},\nabla\varphi^{\beta})\bar{R}^1_{\beta\alpha
v}=-\tau^1g(\nabla\varphi^2,\nabla\varphi^2)\bar{R}^1_{221}-\tau^2g(\nabla\varphi^1,\nabla\varphi^2)\bar{R}^1_{212}\\\notag
=&&-(m^2+\frac{n^2}{\sigma^2})\lambda\lambda''(\rho)x+(am+\frac{bn}{\sigma^2})\lambda\lambda''(\rho)y,\\\notag
\end{eqnarray}
and
\begin{eqnarray}\label{13}
&&-\tau^vg(\nabla\varphi^{\alpha},\nabla\varphi^{\beta})\bar{R}^2_{\beta\alpha
v}=-\tau^1g(\nabla\varphi^2,\nabla\varphi^1)\bar{R}^2_{121}-\tau^2g(\nabla\varphi^1,\nabla\varphi^1)\bar{R}^2_{112}\\\notag
=&&(am+\frac{bn}{\sigma^2})\frac{\lambda''(\rho)}{\lambda}x-(a^2+\frac{b^2}{\sigma^2})\frac{\lambda''(\rho)}{\lambda}y.\\\notag
\end{eqnarray}
Substituting (\ref{12})$\sim$ (\ref{13}) into (\ref{BI3}) we conclude
that $\varphi$ is biharmonic if and only if it solves the system
\begin{equation}\label{pl5}
\begin{cases}
x_{rr}+\frac{x_{\theta\theta}}{\sigma^2}+\frac{\sigma'}{\sigma}x_r-(m^2+\frac{n^2}{\sigma^2})\left(\lambda\lambda'(\rho)\right)'(\rho)
x\\-2m\lambda\lambda'(\rho)y_r-\frac{2n\lambda\lambda'(\rho)}{\sigma^2}y_\theta- \left(\frac{m\sigma'\lambda\lambda'(\rho)}{\sigma}+2(am+\frac{bn}{\sigma^2})\lambda'^2(\rho)\right)y=0,\\
y_{rr}+\frac{y_{\theta\theta}}{\sigma^2}+\left(\frac{\sigma'}{\sigma}+\frac{2a\lambda'(\rho)}{\lambda}\right)y_{r}
+\frac{2b\lambda'(\rho)}{\sigma^2\lambda}y_\theta+\left(\frac{a\sigma'\lambda'(\rho)}{\sigma\lambda}
-(m^2+\frac{n^2}{\sigma^2})\lambda'^2(\rho)\right)y\\
+\frac{2m\lambda'(\rho)}{\lambda}x_r+\frac{2n\lambda'(\rho)}{\sigma^2\lambda}x_\theta
+\left(\frac{m\sigma'\lambda'(\rho)}{\sigma\lambda}
+2(am+\frac{bn}{\sigma^2})\frac{\lambda''(\rho)}{\lambda}\right)x=0,\\
x=\tau^1=a\frac{\sigma'}{\sigma}-(m^2+\frac{n^2}{\sigma^2})\lambda\lambda'(\rho),\\
y=\tau^2=m\frac{\sigma'}{\sigma}+2(am+\frac{bn}{\sigma^2})\frac{\lambda'(\rho)}{\lambda},
\end{cases}
\end{equation}
which is equivalent to the system (\ref{p1}). Thus, the lemma follows.
\end{proof}
Now we are ready to prove the following theorem that gives a classification of all biharmonic maps in a large family that includes most examples mentioned in Section 2.
\begin{theorem}\label{T2S2}
The map $\varphi: (T^2, dr^2+d\theta^2)\longrightarrow (S^2, d\rho^2+\sin^2\rho d\phi^2)$ from a flat torus into a $2$-sphere with $\varphi (r, \theta)=(ar+b\theta+c, mr+n\theta+l)$ is biharmonic if and only if one of the following cases happens\\
$(A)$ $a=b=0$ and $c=\frac{\pi}{2}$. In this case, the map $\varphi (r, \theta)=(\pi/2, mr+n\theta+l)$ is actually a harmonic map,\\
$(B)$ $m=n=0$. In this case, the map $\varphi (r, \theta)=(ar+b\theta+c, l)$ is actually a harmonic map, or\\
$(C)$ $a=b=0$, $m^2+n^2\neq 0$ and $c=\frac{\pi}{4}$, or, $c=\frac{3\pi}{4}$. In this case, the map $\varphi (r, \theta)=(c, mr+n\theta+l)$ with $c=\frac{\pi}{4}$, or, $c=\frac{3\pi}{4}$ is a proper biharmonic map.\\
\end{theorem}
\begin{proof}
Applying Lemma \ref{PL} with $\sigma (r)=1, \lambda(\rho)=\sin\rho$ we conclude that the map $\varphi: (S^1\times S^1, dr^2+d\theta^2)\longrightarrow (S^2, d\rho^2+\sin^2\rho d\phi^2)$ from a flat torus into a $2$-sphere with $\varphi (r, \theta)=(ar+b\theta+c, mr+n\theta+l)$ is biharmonic if and only if
\begin{equation}\label{FT2}
\begin{cases}
 x_{rr}+x_{\theta\theta}-(m^2+n^2)(\cos2 \rho)
x -2m(\sin \rho \cos\rho) y_r\\\
 -2n(\sin \rho \cos\rho)y_\theta- 2(am+bn)(\cos^2 \rho) \,y=0,\\
 y_{rr}+y_{\theta\theta}+2a(\cot\rho)y_{r}
+2b(\cot\rho)y_\theta-(m^2+n^2)(\cos^2 \rho)\,y\\
 +2m(\cot\rho)x_r+2n(\cot\rho)x_\theta
-2(am+bn)x=0,\\
 x=\tau^1=-\frac{1}{2}(m^2+n^2)\sin(2\rho),\\
 y=\tau^2=2(am+bn)\cot\rho.
\end{cases}
\end{equation}

Substituting the last two equations into the first two equations of (\ref{FT2}) we see that the biharmonicity of the map $\varphi$ becomes
\begin{equation}\label{GD1}
\begin{cases}
[4(m^2+n^2)(a^2+b^2)+(m^2+n^2)^2\cos (2\rho)+4(am+bn)^2]\sin \rho\cos \rho=0,\\
4(m^2+n^2)(am+bn)\frac{\cos \rho}{\sin\rho}\cos (2\rho)=0.
\end{cases}
\end{equation}

Noting that $0<\rho=ar+b\theta+c< \pi$ (and hence) $\sin\rho\ne 0$ we conclude that Equation (\ref{GD1}) is equivalent to
\begin{equation}\label{GD2}
\begin{cases}
[4(m^2+n^2)(a^2+b^2)+(m^2+n^2)^2\cos (2\rho)+4(am+bn)^2]\cos\rho=0,\\
4(m^2+n^2)(am+bn)\cos \rho\cos (2\rho)=0.
\end{cases}
\end{equation}

We solve Equation (\ref{GD2}) by considering the following cases,\\
{\bf Case I:} $\cos \rho=0$. This means that $\cos (ar+b\theta+c)=0$ for any $r, \theta \in \r$. This, together with $0<\rho=ar+b\theta+c< \pi$, implies that $a=b=0, c=\pi/2$. Noting that the components of the tension field of the map $\varphi$ are given by the last two equations of (\ref{FT2}) we conclude that the solutions ($a=b=0, c=\pi/2, m, n, l \in \r$) given in this case are actually harmonic maps. From this we obtain case (A).\\

{\bf Case II:} $\cos \rho\ne 0$. In this case, the biharmonicity of the map $\varphi$ is equivalent to
\begin{equation}\label{GD3}
\begin{cases}
4(m^2+n^2)(a^2+b^2)+(m^2+n^2)^2\cos (2\rho)+4(am+bn)^2=0,\\
4(m^2+n^2)(am+bn)\cos (2\rho)=0.
\end{cases}
\end{equation}

If $m=n=0$, then, $m=n=0, a, b, c, l\in \r$ are solutions of (\ref{GD3}) and we see from the last two equations of (\ref{FT2}) that the maps given by these solutions are actually harmonic maps. This gives case (B).\\

If otherwise, i.e., $m^2+n^2\ne 0$, then it follows from (\ref{GD3}) that the map $\varphi$ is biharmonic if and only if
\begin{equation}\label{GD4}
\begin{cases}
\cos (2\rho)=-\frac{4(m^2+n^2)(a^2+b^2)+4(am+bn)^2}{(m^2+n^2)^2},\\
(am+bn)\cos (2\rho)=0.
\end{cases}
\end{equation}
If $am+bn\ne 0$, then  we can easily check that Equation (\ref{GD4}) has no solution. If $am+bn=0$, then the map $\varphi$ is biharmonic if and only if
\begin{equation}\label{GD5}
\begin{cases}
\cos (2\rho)=-\frac{4(a^2+b^2)}{m^2+n^2},\\
am+bn=0.
\end{cases}
\end{equation}

Since the first equation of (\ref{GD5}) means that $\cos (2ar+2b\theta+2c)=-\frac{4(a^2+b^2)}{m^2+n^2}$ for any $r, \theta \in \r$, we conclude that $a=b=0$ and hence $\cos (2c)=0$. Recalling that $2\rho=2c\in (0, 2\pi)$ we obtain solutions $c=\pi/4, 3\pi/4$. In these cases, $\cos \rho=\cos c\ne 0$. It follows from the third equation of (\ref{FT2}) that the first component of the tension field $\tau^1=-\frac{1}{2}(m^2+n^2)\sin(2\rho)=\pm \frac{1}{2}(m^2+n^2)\ne 0$. Thus, the biharmonic maps $\varphi(r, \theta)=(\pi/4, mr+n\theta +l),\; \varphi(r, \theta)=(3\pi/4, mr+n\theta +l)$ are proper biharmonic maps. From this we obtain case (C).\\

Summarizing the above results we obtain the theorem.
\end{proof}

As immediate consequences of Theorem \ref{T2S2}, we have the following corollaries.
\begin{corollary}\label{F4pi}
For $m^2+n^2\neq 0$, the map $(T^2, dr^2+d\theta^2)\longrightarrow (S^2, d\rho^2+\sin^2\rho d\phi^2)$ with $\varphi (r, \theta)=(\pm \frac{1}{\sqrt{2}},  \frac{1}{\sqrt{2}} e^{(mr+n\theta+l)i})$ is a proper biharmonic map.\\
In particular, the compositions of the family of harmonic embeddings $F_s: (T^2, dr^2+d\theta^2)\longrightarrow S^3$,
$F_s(r,\theta)=(e^{ir}\cos s , e^{i\theta}\sin s )$ followed by the Hopf fibration $H: S^3\longrightarrow S^2$,
\begin{eqnarray}\notag
&& \varphi_s=H\circ F_s: (T^2, dr^2+d\theta^2)\longrightarrow (S^2, d\rho^2+\sin^2\rho d\phi^2), \\\notag && \varphi_s(r, \theta)=(\rho(r, \theta), \phi(r, \theta), {\rm with}\;\; \begin{cases}
\rho(r, \theta)=2s, \; s={\rm constant}, \\\phi(r, \theta)=r-\theta,
\end{cases}
\end{eqnarray}
are proper biharmonic maps if and only if $s=\pi/8$ or $s=3\pi/8$.
\end{corollary}
\begin{remark}
Note that Corollary \ref{F4pi} provides many new examples of proper biharmonic maps from a flat torus into a sphere, including special cases $\varphi: (S^1\times S^1, dr^2+d\theta^2)\longrightarrow (S^2, g^{S^2})$ with $\varphi (r, \theta)=(\pm \frac{1}{\sqrt{2}},  \frac{1}{\sqrt{2}}\cos\theta,  \frac{1}{\sqrt{2}}\sin \theta )$. These special cases, up to an isometry of the target sphere, are the same maps obtained in \cite{Ou4} by construction of orthogonal multiplication of complex numbers (See \cite{Ou4}, Theorem 2.2 for details). These special cases of proper biharmonic map were also known as special solutions to the biharmonic equation for rotationally symmetric maps from a flat torus into a $2$-sphere (See \cite{MR} and \cite{WOY}).
\end{remark}
\begin{corollary}
The map $ \varphi=H\circ f:(T^2, dr^2+d\theta^2)\longrightarrow (S^2, d\rho^2+\sin^2\rho d\phi^2)$
\begin{eqnarray}\notag
\varphi(r, \theta)=(\rho(r, \theta), \phi(r, \theta)), {\rm with}\; \begin{cases}
\rho(r,\theta)=-r-\theta, \\\phi(r,\theta)=\theta,
\end{cases}
\end{eqnarray}
constructed in Example \ref{Ex4} is neither a harmonic map nor a biharmonic map.
\end{corollary}
\begin{corollary}
The Gauss map of the torus $X: T^2 \longrightarrow \r^3$, $X(r, \theta)=( a\sin r,\; (b+a\cos r)\cos \theta,\; (b+a\cos r)\sin \theta)$, viewed as a map $(T^2, dr^2+d\theta^2)\longrightarrow (S^2, d\rho^2+\sin^2\rho d\phi^2)$ from a flat torus into a sphere, is neither harmonic nor biharmonic.
\end{corollary}
\begin{corollary}
Let $f: (T^2, dr^2+d\theta^2)\longrightarrow S^3$ with $f(r,\theta)=(\cos (kr)e^{im\theta} , \sin (kr) e^{in\theta})$ be the family of immersions studied in \cite{La} . Then, the family of maps from a torus into a $2$-sphere defined by
the construction via Hopf fibration
\begin{eqnarray}\notag
&& \varphi=H\circ f: (T^2, dr^2+d\theta^2)\longrightarrow (S^2, d\rho^2+\sin^2\rho d\phi^2), \\\notag &&\phi(r,\theta)=(\rho(r,\theta),\phi(r,\theta)), \; {\rm with}\;\;
\begin{cases}
\rho(r,\theta)=2kr, \\\phi(r,\theta)=(m-n)\theta.
\end{cases}
\end{eqnarray}
contains no proper biharmonic map.
\end{corollary}
\section{Biharmonic maps from a non-flat torus into a $2$-sphere}
In this section, we will study biharmonic maps from a non-flat torus into a 2-sphere. The non-flat metric on the torus we consider is $dr^2+(k+\cos r)^2 d\theta^2$ which is homothetic to the induced metric $g_{T}=a^2dr^2+(b+a\cos r)^2d\theta^2$ (See Example \ref{Ex5} for details) from the standard embedding $X: T^2 \longrightarrow \r^3$, $X(r, \theta)=( a\sin r,\; (b+a\cos r)\cos \theta,\; (b+a\cos r)\sin \theta)$.
\begin{theorem}\label{NTS}
The map from a non-flat torus into a $2$-sphere, $\varphi: (T^2, dr^2+(k+\cos r)^2d\theta^2)\longrightarrow (S^2, d\rho^2+\cos^2\rho d\phi^2)$ $(k>1)$ with $\varphi (r, \theta)=(ar+b\theta+c, mr+n\theta+l)$ is biharmonic if and only if it is harmonic.
\end{theorem}
\begin{proof}
Let $\varphi:(T^2,dr^2+(k+\cos r)^2 d\theta^2)\longrightarrow
(S^2,d\rho^2+\cos^2\rho d\phi^2)$ $(k>1)$ with $\varphi(r,\theta)=(ar+b\theta+c, mr+n\theta+l)$. Using Lemma \ref{PL} with $\sigma (r)=k+\cos r,\;\lambda(\rho)=\cos\rho$, we see that $\varphi$ is biharmonic if and only if it solves the system
\begin{equation}\label{NTSS1}
\begin{cases}
\frac{x_{\theta\theta}}{(k+\cos r)^2}+x_{rr}-\frac{\sin r}{k+\cos r}x_r+(m^2+\frac{n^2}{(k+\cos r)^2})\cos2\rho\;
x\\+\left(2my_r+\frac{2n}{(k+\cos r)^2}y_\theta+y^2\right)\sin\rho\cos\rho=0,\\
\frac{y_{\theta\theta}}{(k+\cos r)^2}+y_{rr}-\frac{\sin r}{k+\cos r}y_{r}-2\left(ay_{r}
+mx_r+\frac{by_\theta+nx_\theta}{(k+\cos r)^2}
\right)\frac{\sin\rho}{\cos\rho}
\\+2\left(\frac{m\sin r\sin\rho}{(k+\cos r)\cos\rho}
-(am+\frac{bn}{(k+\cos r)^2})\frac{\cos2\rho}{\cos^2\rho}\right)x=0,\\
x=\tau^1=-\frac{a\sin r}{k+\cos r}+\frac{1}{2}(m^2+\frac{n^2}{(k+\cos r)^2})\sin2\rho,\\
y=\tau^2=-\frac{m\sin r}{k+\cos r}-2(am+\frac{bn}{(k+\cos r)^2})\tan\rho.
\end{cases}
\end{equation}
A straightforward computation using the last two equations of (\ref{NTSS1}) yields
\begin{equation}\label{NTSS2}
\begin{cases}
x_r=-\frac{ak\cos r+a}{(k+\cos r)^2}+\frac{n^2\sin r\sin2\rho}{(k+\cos r)^3}+a(m^2+\frac{n^2}{(k+\cos r)^2})\cos2\rho,\\
x_{rr}=-\frac{a(2-k^2)\sin r+ak\sin r\cos r-4an^2\sin r \cos 2\rho}{(k+\cos r)^3}\\-\frac{2a^2m^2(k+\cos r)^4+2a^2n^2(k+\cos r)^2-n^2(k\cos r+\cos^2r+3\sin^2r)}{(k+\cos r)^4}\sin2\rho,\\
x_\theta=b(m^2+\frac{n^2}{(k+\cos r)^2})\cos2\rho,\\
x_{\theta\theta}=-2b^2(m^2+\frac{n^2}{(k+\cos r)^2})\sin2\rho,\\
y_r=-\frac{mk\cos r+m}{(k+\cos r)^2}-\frac{4bn\sin r\tan\rho}{(k+\cos r)^3}-\frac{2a(am+\frac{bn}{(k+\cos r)^2})}{\cos^2\rho},\\
y_{rr}=-\frac{m(2-k^2)\sin r+mk\sin r\cos r}{(k+\cos r)^3}-\frac{4bn(k\cos r+\cos^2r+3\sin^2r)\tan\rho}{(k+\cos r)^4}\\-\frac{8abn\sin r}{(k+\cos r)^3\cos^2\rho}-\frac{4a^2(am+\frac{bn}{(k+\cos r)^2})\sin\rho}{\cos^3\rho},\\
y_\theta=-\frac{2b(am+\frac{bn}{(k+\cos r)^2})}{\cos^2\rho},\\
y_{\theta\theta}=\frac{4b^2(am+\frac{bn}{(k+\cos r)^2})\sin\rho}{\cos^3\rho}.
\end{cases}
\end{equation}
Substituting (\ref{NTSS2}) into the first equation of (\ref{NTSS1}) we have
\begin{equation}\label{NTSS3}
\begin{cases}
\frac{(a(k^2-1)-2bmn)\sin r}{(k+\cos r)^3}+\frac{2am^2\sin r}{k+\cos r}+[\frac{n^2(k\cos r+\sin^2r+1)}{(k+\cos r)^4}+\frac{\frac{m^2}{2}\sin^2r-m^2k\cos r-m^2}{(k+\cos r)^2}\\-2(am+\frac{bn}{(k+\cos r)^2})^2-2(m^2+\frac{n^2}{(k+\cos r)^2})(a^2+\frac{b^2}{(k+\cos r)^2})]\sin2\rho\\+[\frac{2an^2\sin r}{(k+\cos r)^3}-\frac{4am^2\sin r}{k+\cos r}+\frac{2bmn\sin r}{(k+\cos r)^3}]\cos2\rho+\frac{1}{4}(m^2+\frac{n^2}{(k+\cos r)^2})^2\sin4\rho=0,\\
\rho=ar+b\theta+c.
\end{cases}
\end{equation}
We will solve Equation (\ref{NTSS3}) by the following cases:\\
{\bf Case $(I)$:} $b\neq0$.\\
In this case, using the assumption that $b\neq0$ and the fact that the functions $1,\;\sin2\rho,\;\cos2\rho$ and $\sin4\rho$ are linearly independent as functions of  variable $\theta$, we conclude from (\ref{NTSS3}) that
\begin{equation}\label{NTSS5}
\begin{cases}
\frac{(a(k^2-1)-2bmn)\sin r}{(k+\cos r)^3}+\frac{2am^2\sin r}{k+\cos r}=0,\\
\frac{n^2(k\cos r+\sin^2r+1)}{(k+\cos r)^4}+\frac{\frac{m^2}{2}\sin^2r-m^2k\cos r-m^2}{(k+\cos r)^2}\\-2(am+\frac{bn}{(k+\cos r)^2})^2-2(m^2+\frac{n^2}{(k+\cos r)^2})(a^2+\frac{b^2}{(k+\cos r)^2})=0,\\
\frac{2an^2\sin r}{(k+\cos r)^3}-\frac{4am^2\sin r}{k+\cos r}+\frac{2bmn\sin r}{(k+\cos r)^3}=0,\\
\frac{1}{4}(m^2+\frac{n^2}{(k+\cos r)^2})^2=0.
\end{cases}
\end{equation}

From the fourth and the first equation of (\ref{NTSS5}) we have $a=m=n=0$. In this case, we use the last two equations in (\ref{NTSS1}) to conclude that the components of the tension field $x=\tau^1=0,\;
y=\tau^2=0$. This implies that the map $\varphi$ is actually harmonic.\\
{\bf Case $(II)$:} $b=0$ and $a=0$.\\
In this case, substituting $a=b=0$ and (\ref{NTSS2}) into the second equation of (\ref{NTSS1}) we obtain
\begin{equation}\label{NTSS7}
\begin{cases}
\frac{m(k^2-1)\sin r}{(k+\cos r)^3}-\frac{2mn^2\sin r}{(k+\cos r)^3}\sin^2\rho+\frac{2m^3\sin r}{k+\cos r}\sin^2\rho=0,\\
\rho=c.
\end{cases}
\end{equation}
If $c=0$, then $\sin \rho=0$ and (\ref{NTSS7}) reduces to $m(k^2-1)=0$ which implies $m=0$ since $k>1$ by assumption. In this case, the last two equations in (\ref{NTSS1}) imply that $\tau(\varphi)=0$ and hence $\varphi$ is harmonic.\\
If $c\ne 0$, then (\ref{NTSS7}) is equivalent to
\begin{equation}
\begin{cases}
m[(k^2-1)+2(m^2k^2-n^2)\sin^2\rho+4m^2k\sin^2\rho\cos r+2m^2\sin^2\rho\cos^2r]=0,\\
\rho=c\ne 0.
\end{cases}
\end{equation}
If $m\neq0$, we have
\begin{equation}\label{NTSS8}
\begin{cases}
k^2-1+2(m^2k^2-n^2)\sin^2\rho+4m^2k\sin^2\rho\cos r+2m^2\sin^2\rho\cos^2r=0,\\
\rho=c\ne 0.
\end{cases}
\end{equation}
Note that $k>1$ and the functions $1,\;\cos r,\;\cos^2r$ are linearly independent, Equation (\ref{NTSS8}) implies that $ \sin^2\rho=0$ and reduces to $k^2-1=0$, a contradiction. \\
If otherwise, i.e., $m=0$, substituting $a=b=m=0$ and (\ref{NTSS2}) into the first equation of (\ref{NTSS1}) we have
\begin{equation}
\begin{cases}
\frac{n^2\sin2\rho}{2(k+\cos r)^4}(2k\cos r-2\cos^2r+4+n^2\cos2\rho)=0,\\
\rho=c\ne 0.
\end{cases}
\end{equation}
The above equation implies that $n^2\sin2\rho=0$, i.e., $n=0$. It follows that $x=\tau^1=0,\;y=\tau^2=0$, meaning that the map $\varphi$ is harmonic in this case.\\

{\bf Case $(III)$:} $b=0$, $a\neq0$. We will show that Equation (\ref{NTSS3}) has no solution in this case.\\

Multiplying $(k+\cos r)^4$ to both sides of Equation ( \ref{NTSS3}) and
simplifying the resulting equation by using the product-to-sum formulas we have
\begin{eqnarray}\notag 
&& a_0\sin(2ar+2c)
+\sum_{i=1}^4\frac{a_i\pm b_i}{2}\sin[(2a \pm i)r+2c]\\\notag
&&+c_0\sin(4ar+4c)+\sum_{i=1}^4\frac{c_i}{2}\sin[(4a \pm i)r+4c]\\\notag
&&+d_1\sin r+d_2\sin2 r+d_3\sin 3r+d_4\sin 4r=0,
\end{eqnarray}
or, equivalently
\begin{eqnarray}\notag \label{NTSS31}
&& a_0\cos(2c)\sin(2ar)+a_0\sin (2c)\cos(2ar)\\\notag
&&+\cos(2c) \sum_{i=1}^4\frac{a_i\pm b_i}{2}\sin(2a \pm i)r+\sin (2c)\sum_{i=1}^4\frac{a_i\pm b_i}{2}\cos(2a \pm i)r\\
&&+c_0\cos (4c)\sin(4ar)+c_0\sin(4c)\cos (4ar)\\\notag
&&+\cos (4c) \sum_{i=1}^4\frac{c_i}{2}\sin(4a \pm i)r +\sin (4c) \sum_{i=1}^4\frac{c_i}{2}\cos(4a \pm i)r\\\notag
&&+d_1\sin r+d_2\sin2 r+d_3\sin 3r+d_4\sin 4r=0,
\end{eqnarray}
where
\begin{eqnarray}\label{NTSS30}
a_0 &=& -7m^2k^2/4-2a^2n^2k^2-a^2n^2-7m^2/16-4a^2m^2k^4\\\notag &&-12a^2m^2k^2-3a^2m^2/2+3n^2/2,\\
a_1 &=& -5m^2k/2-m^2k^3-4a^2n^2k-16a^2m^2k^3-12a^2m^2k+n^2k,\\
a_2 &=& -5m^2k^2/4-a^2n^2-m^2/2-12a^2m^2k^2-2a^2m^2-n^2/2,\\
a_3 &=& -m^2k/2-4a^2m^2k, \;\;\;\; a_4 =-m^2/16-a^2m^2/2,\\
b_1 &=& -4am^2k^3-3am^2k+2an^2k,\;\;b_2 = -6am^2k-am^2+an^2,\\
b_3 &=& -3am^2k, \;\;\;\;\; b_4 = -am^2/2,\\
c_0 &=&m^4k^4/4+3m^4k^2/4+3m^4/32+m^2n^2k^2/2+m^2n^2/4+n^4/4\\
c_1 &=& m^4k^3+3m^4k/4+m^2n^2k, \;\; c_2 = 3m^4k^2/4+m^2n^2/4+m^4/8,\\
c_3 &=&m^4k/4, \;\;\;\; c_4 =m^4/32,\\
d_1 &=& a(k^2-1)k+am^2(2k^3+3k/2),\\
d_2 &=&a(k^2-1)/2+am^2(3k^2+1/2),\\
d_3 &=&3am^2k/2, \;\;\;\; d_4 = am^2/4.
\end{eqnarray}

We observe that the $40$ trigonometric functions appear in the linear combination on the left hand side of Equation (\ref{NTSS31} ) are linearly independent for the values of $a$ that produce no like terms among them. Note also that even in the case the values of $a$ produce like terms, we can collect the like terms so that the resulting set of functions are linearly independent. 

{\bf Case $(A)$:} For those values of $a$ that do not turn any of $\sin(2ar)$, \\$\sin(2a \pm i)r, \sin(4ar), \sin(4a \pm i)r (i=1, 2, 3, 4) $ into a like term of  $\sin r, \sin2 r, \sin 3r$, \\$\sin 4r$. In this case, we have all coefficients vanish, including $d_1=d_4=0$, which imply $a=0$, a contradiction.\\

{\bf Case $(B)$:} For those values of $a$ that turn one of $\sin(2ar), \sin(2a \pm i)r, \sin(4ar), $\\$\sin(4a \pm i)r (i=1, 2, 3, 4) $ into a like term to one of $\sin r, \sin2 r, \sin 3r, \sin 4r$. We can check that the only values of $a$ that turn one of $\sin(2ar), \sin(2a \pm i)r, \sin(4ar), \sin(4a \pm i)r $ $ (i=1, 2, 3, 4) $ into a like term to one of  $\sin r, \sin2 r, \sin 3r$, \\$\sin 4r$ are the following:
\begin{equation}\label{AV}
a=\pm j, \pm \frac{2j-1}{2}, \pm \frac{2j-1}{4}, j=1, 2, 3, 4.
\end{equation}

It is not difficult to check that none of positive values of $a$ given in (\ref{AV}) can produce like term of $\sin (4a+4)r$ and none of the negative values of $a$ given in (\ref{AV}) can produce like term of $\sin (4a-4)r$. It follows that for each value of $a$ given in (\ref{AV}), we can, after a possible collecting of like terms in (\ref{NTSS31} ), use the linear independence of the resulting set of functions to conclude that $c_4= \frac{m^4}{32}=0$. This implies that $m=0$ for any values of $a$ given in (\ref{AV}). Substituting $m=0$ into Equation (\ref{NTSS31} ) we have
\begin{eqnarray}\notag \label{m01}
&& \sin (2c)\{a_0\cos(2ar)+\frac{a_1+b_1}{2}\cos[(2a+1)r]+\frac{a_1-b_1}{2}\cos[(2a-1)r]\\
&& +\frac{a_2+b_2}{2}\cos[(2a+2)r]+\frac{a_2-b_2}{2}\cos[(2a-2)r]\}+c_0\sin(4c) \cos(4ar)\\\notag
&& +\cos (2c)\{a_0\sin(2ar)+\frac{a_1+b_1}{2}\sin[(2a+1)r]+\frac{a_1-b_1}{2}\sin[(2a-1)r]\\\notag
&& +\frac{a_2+b_2}{2}\sin [(2a+2)r]+\frac{a_2-b_2}{2}\sin[(2a-2)r]\}+c_0\cos(4c) \sin(4ar)\\\notag
&&+d_1\sin r+d_2\sin2 r=0,
\end{eqnarray}
where
\begin{eqnarray}\label{NTSS30}\label{a0}
a_0 &=& -2a^2n^2k^2-a^2n^2+3n^2/2,\;\;\;\;\;\;\;a_1 = -4a^2n^2k+n^2k, \\ a_2 &=& -a^2n^2-n^2/2,\;\;\;\;\;
b_1 = 2an^2k,\;\;b_2 = an^2,\;\;\;\;\;\;\;\; c_0 =\frac{n^4}{4},\\
d_1 &=& a(k^2-1)k,\;\;\;\;d_2 = a(k^2-1)/2.
\end{eqnarray}

We can check that the only values of $a$ that turn $\sin (2ar)$ into a like term to one of $\sin(2a \pm i)r, (i=1,2), \sin(4ar),  \sin r, \sin2 r$ are $a= \pm 1,  \pm 1/2, \pm 1/4$. \\

For the case $a\ne \pm 1,  \pm 1/2, \pm 1/4$, $\sin (2ar)$ is not a like term to any other term and hence we can use linear independence of functions in Equation (\ref{m01}) to conclude that $a_0=0$. Using (\ref{a0}) we have either $3/2-a^2(2k^2+1)=0$ or $n=0$. The first case gives the values of $a$ contradicting those given by (\ref{AV}). The latter case, i.e., $n=0$, implies that $a_0=a_1=a_2=b_1=b_2=c_0=0$. Substituting these into (\ref{m01}) we have $d_1=d_2=0$ which imply $a=0$, a contradiction.\\

For the case $a= \pm 1,  \pm 1/2$, or $\pm 1/4$, we first check that for $a=\pm 1/4$, we have $d_2=0$, a contradiction. Secondly, we  check that for $a= \pm 1,  \pm 1/2$, we have $a_1\pm b_1=0, a_2\pm b_2=0$ respectively. A further checking shows that in each of these four cases, we have  $n=0$, which, as we have seen in the above argument, will lead to a contradiction. This ends the proof that the biharmonic map equation has no solution in Case B.

Summarizing the results in Cases A and B we conclude that for the case $b=0, a\ne 0$, the biharmonic map equations have no solution.\\

Combining the results proved in  Cases (I), (II),  and (III) we obtain the theorem.
\end{proof}

From the froof of Theorem \ref{NTS} we have seen the following
\begin{corollary}
For $a\neq 0$, there exists no biharmonic map in the family of the maps from a non-flat torus into a $2$-sphere,  $\varphi: (T^2, dr^2+(k+\cos r)^2d\theta^2)\longrightarrow (S^2, d\rho^2+\cos^2\rho d\phi^2)$ $(k>1)$ with $\varphi (r, \theta)=(ar+b\theta+c, mr+n\theta+l)$.
\end{corollary}

Noting that the models $(S^2, d\rho^2+\sin^2\rho d\phi^2)$ and $(S^2, d\rho^2+\cos^2\rho d\phi^2)$ of a $2$-sphere are isometric to each other, we can use Theorem \ref{NTS} to have the following classification results.
\begin{corollary}\label{4pi}
The composition of the family of immersions $f_s: (T^2, dr^2+(k+\cos r)^2d\theta^2)\longrightarrow S^3$, $f_s(r,\theta)=(e^{ir}\cos s , e^{i\theta}\sin s )$ followed by the Hopf fiberation $H: S^3\longrightarrow S^2$,
\begin{eqnarray}\notag
&& \varphi_s=H\circ f_s: (T^2, dr^2+(k+\cos r)^2d\theta^2)\longrightarrow (S^2, d\rho^2+\sin^2\rho d\phi^2), \\\notag && \varphi_s(r,\theta)=(\rho(r,\theta), \phi(r,\theta)), {\rm with}\;\; \begin{cases}
\rho(r,\theta)=2s, \; s={\rm constant}, \\\phi(r,\theta)=r-\theta,
\end{cases}
\end{eqnarray}
is neither a harmonic map nor a biharmonic map.
\end{corollary}
\begin{corollary}
The composition of the map $f:(T^2, dr^2+(k+\cos r)^2d\theta^2)\longrightarrow S^3$ with
$f(r,\theta)=\big(\cos\frac{-(r+\theta)}{2} e^{i\frac{x+y}{2}},\;\sin \frac{-(r+\theta)}{2}\,e^{i\frac{r-\theta}{2}}\big)$
given in \cite{Ou4} followed by the construction via Hopf fibration
\begin{eqnarray}\notag
&& \varphi=H\circ f:(T^2, dr^2+(k+\cos r)^2d\theta^2)\longrightarrow (S^2, d\rho^2+\sin^2\rho d\phi^2),\\\notag
&&\varphi(r,\theta)=(\rho(r,\theta), \phi(r,\theta)), {\rm with}\; \begin{cases}
\rho(r,\theta)=-r-\theta, \\\phi(r,\theta)=\theta,
\end{cases}
\end{eqnarray}
is neither a harmonic map nor a biharmonic map.
\end{corollary}
\begin{corollary}
Let $f: (T^2, dr^2+(k+\cos r)^2d\theta^2)\longrightarrow S^3$ with $f(r,\theta)=(\cos (cr)e^{im\theta} , \sin (cr) e^{in\theta})$ be the family of immersions studied in \cite{La} . Then, the family of maps from a torus into a $2$-sphere defined by
the construction via Hopf fibration
\begin{eqnarray}\notag
&& \varphi=H\circ f: (T^2, dr^2+(k+\cos r)^2d\theta^2)\longrightarrow (S^2, d\rho^2+\sin^2\rho d\phi^2), \\\notag &&\phi(r,\theta)=(\rho(r,\theta),\phi(r,\theta)), \; {\rm with}\;\;
\begin{cases}
\rho(r,\theta)=2cr, \\\phi(r,\theta)=(m-n)\theta.
\end{cases}
\end{eqnarray}
is neither a harmonic map nor a biharmonic map.
\end{corollary}
\begin{corollary}
Let $X: T^2 \longrightarrow \r^3$ be an embedding with $X(r, \theta)=( a\sin r,\; (b+a\cos r)\cos \theta,\; (b+a\cos r)\sin \theta)$ with $b>a>0$. Then, the Gauss map $\varphi:(T^2,g_{T}=a^2dr^2+(b+a\cos r)^2d\theta^2) \longrightarrow (S^2,h=dr^2+\cos^2 rd\theta^2)$ with $\varphi(r,\theta)=(r,\theta)$ is neither harmonic nor biharmonic.
\end{corollary}

\end{document}